\documentclass[12pt]{amsart}
\usepackage{amsfonts}
\usepackage{amssymb,amsthm,amsmath}
\usepackage[mathscr]{eucal}
\usepackage[notcite,notref]{}
\usepackage[latin1]{inputenc}

\numberwithin{equation}{section}
\theoremstyle{plain}
\newtheorem{theorem}{Theorem}
\newtheorem{lemma}[theorem]{Lemma}
\newtheorem{corollary}[theorem]{Corollary}

\theoremstyle{definition}
\newtheorem{definition}[theorem]{Definition}

\theoremstyle{remark}
\newtheorem{remark}[theorem]{Remark}
\newtheorem{example}[theorem]{Example}

\title[Regularity of the Hardy-Littlewood maximal operator]{Regularity of the Hardy-Littlewood maximal operator on
 block decreasing functions}

\author{J. M. Aldaz}
\address{PERMANENT ADDRESS: Departamento de Matem\'aticas y Computaci\'on,
Universidad  de La Rioja, 26004 Logro\~no, La Rioja, Spain.}
\email{jesus.munarrizaldaz@unirioja.es}

\address{CURRENT ADDRESS: Departamento de Matem\'aticas,
Universidad  Aut\'onoma de Madrid, Cantoblanco 28049, Madrid, Spain.}
\email{jesus.munarriz@uam.es}

\author{F.J. P\'erez L\'azaro}
\thanks{Research supported in part by grant MTM2006-13000-C03-03 of the DGI, Spain}
\address{Departamento de Matem\'aticas y Computaci\'on,
Universidad de La Rioja, Edificio J.~L.~Vives, Calle Luis de Ulloa
s/n, 26004 Logro\~no, Spain} \email{javier.perezl@unirioja.es}


\begin{document}
\baselineskip=17pt

\maketitle
\begin{abstract}
We study the Hardy-Littlewood maximal operator defined via an
unconditional norm, acting on block decreasing functions. We show that
the uncentered maximal operator maps block decreasing functions of special bounded variation to functions with
integrable distributional derivatives,
thus improving  their regularity. In the special case of the
maximal operator
defined by the $\ell_\infty$-norm, that is, by averaging over cubes,
the result extends to block decreasing functions of bounded variation, not necessarily
special.
\end{abstract}

\section{Introduction}

The usefulness of the Hardy-Littlewood
maximal function $M$ stems basically from two facts:

1) It is
larger than
the given function, since $|f|\le Mf$ a.e.,
but it is not too large, since $\|Mf\|_p \le c_p \|f\|_p$ for  $1 < p \le \infty$, while on $L^1$, $M$
satisfies a weak type $(1,1)$ inequality.

2) It is more regular than the original function:
If $f$ is measurable, then $Mf$ is lower semicontinuous.

The fact that $Mf$ controls $f$ and its averages over balls (by
definition) together with its $L^p$ boundedness, leads to its frequent
use in chains of inequalities, while its lower semicontinuity
allows one to decompose its level sets using dyadic cubes. This is
the basis of the often applied Calder\'on-Zygmund decomposition: Utilize
$Mf$ as a proxy for $f$, splitting the open set $\{Mf > t\}$ into suitable
disjoint cubes. This might be impossible to do directly with $\{f > t\}$,
since in principle this set is merely measurable.

Regarding derivatives,
the study of the regularity properties of the  Hardy-Littlewood
maximal function is
much more recent. It
was initiated in \cite{Ki} by Juha Kinnunen, who
proved that the centered maximal operator is bounded on the
Sobolev spaces $W^{1,p}(\mathbb{R}^{d})$ for $1<p\le\infty$.  Since then,
a good deal of work has been done within
this line of research, cf. for instance, \cite{KiLi}, \cite{HaOn},
 \cite{KiSa},  \cite{Lu}, \cite{Bu}, \cite{Ko1},
\cite{Ko2}, \cite{Ta}. The overall emerging pattern,
concerning regularity,  seems to
be that the worse $f$ is, the greater the improvement of $Mf$ when
compared to $f$. In fact,
if the functions are ``good", there may not be any improvement at all.
For instance, the maximal function of a $C^1$ function need not be $C^1$,
while the maximal function of a Lipschitz or an $\alpha$-H\"older function
will be Lipschitz or $\alpha$-H\"older, but in general no better
than that, though constants will be lowered (so there is some
quantitative improvement), cf. \cite{ACP}.

In the present paper only the uncentered maximal function will be
considered, since it has better regularity properties than its
centered relative. A model example of this fact is the following:
Let $f$ be the characteristic function of the unit interval in the
real line. Then both $f$ and the centered maximal function of
$f$ are discontinuous functions at $0$ and $1$, while $Mf$ is Lipschitz
on $\mathbb{R}$ with constant 1.

Here, the dimension $d$ will always be at least two.
The one dimensional case was studied in \cite{AlPe}; there we showed
 that given an arbitrary interval $I\subset \mathbb{R}$, if $f: I\to\Bbb R$ is of bounded variation
and  $Df$ denotes its distributional derivative, then
  $Mf$ is absolutely continuous and  $\| DMf\|_{L^1(I)} \le
|Df|( I)$, where $|Df|$ is the total variation of $Df$ (cf. \cite[Theorem 2.5]{AlPe}). Hence, $M$ improves the regularity of
  $BV$ functions, so, like in the case of the Calder\'on-Zygmund
decomposition, $Mf$ can be used as a proxy for $f$, with
the function $DMf$ replacing the measure $Df$. Along these lines,
a Landau type inequality is presented in \cite[Theorem 5.1]{AlPe}.
Of course, having a function as derivative, instead of a singular measure, makes it possible
to consider $\|DMf\|_{L^p}$ for $p > 1$. In turn, this
suggests the possibility of obtaining inequalities of Gagliardo-Nirenberg-Sobolev  for functions less regular than those in the
Sobolev classes.

Thus, it is interesting to try to find higher dimensional versions
of \cite[Theorem 2.5]{AlPe}. In \cite[Theorem 2.19 and Remark 2.20]{AlPe2}, we
showed that the local maximal function $M_R$ (where the radii of
balls is bounded above by $R$) maps  $BV(\mathbb{R}^d)$ boundedly
into $L^1(\mathbb{R}^d)$, with constant of the order of $\log R$.
In fact, even the local, strong maximal function is bounded
from $BV(\mathbb{R}^d)$
into $L^1(\mathbb{R}^d)$ (with constant of the order of $\log^d R$).
However, the derivative of the local, strong maximal function is
not always
 comparable to $\|f\|_{BV}$, cf. \cite[Theorem 2.21]{AlPe2}, so we have  unboundedness of this operator on $BV$.

Regarding the maximal operator, two questions remain open. First,
whether it regularizes functions in $BV$, so if $f\in BV$, then
$Mf$ is $ACL$ (absolutely continuous on lines, the natural generalization
of absolute continuity to $d > 1$), and second, whether the size
of $DMf$ is not too large, i.e., there exists a constant $c$ such
that $|DMf|(\mathbb{R}^d) \le c\|f\|_{BV}$. At this point both
questions seem to be  intractable. We mention,
after  recalling that $W^{1,1}(\mathbb{R}^d)\subset BV (\mathbb{R}^d)$,
 a related and simpler
question from \cite[Question 1]{HaOn}: Is the operator
$f\mapsto |\nabla Mf|$ bounded from the Sobolev space $W^{1,1}(\mathbb{R}^d)$
to $L^1(\mathbb{R}^d)$?

The results presented here are
obtained by restricting ourselves to a smaller class of functions:
The block decreasing (or unconditional decreasing) functions of bounded
variation (which in particular, contain the $W^{1,1}(\mathbb{R}^d)$
block decreasing functions). For these functions Haj\l asz and Onninen's question has a positive answer.

In general, the balls we use when defining the
maximal operator are unconditional.  Specializing to cubes,
we obtain stronger results.

More precisely, let  $f\ge 0$ be block decreasing. We show that  if $Mf$
is defined using unconditional balls, then
 the variation of $f$ controls the variation of $Mf$ (Theorem \ref{t1}).
If $f$ has finite variation, then $Mf$ is continuous a.e.
with respect to the $(d-1)$-dimensional Hausdorff measure
(no unconditionally needed here, cf. Theorem \ref{t2}). Further assumptions on $f$ lead to better results: If $f$ is also of
special bounded variation (so the derivative $Df$ has no Cantor part),
then $Mf$ has an integrable  weak gradient (cf. Theorem \ref{sbv}).
When the maximal function is defined using cubes, $Mf$ has a weak gradient
even if $Df$ has a nontrivial Cantor part (cf. Theorem \ref{cubes}). Identical results hold for the local maximal operator
$M_R$. Using the fact that  $M_R$  maps  $BV(\mathbb{R}^d)$ boundedly
into $L^1(\mathbb{R}^d)$, we obtain boundedness results for $M_R$
from the non-negative block decreasing functions into the functions of bounded
variation, cf. Corollary \ref{cor1}.

\section{Definitions and results}

\begin{definition}  A function $f : \mathbb{R}^d\to [-\infty, \infty]$ is
{\em unconditional} if  for all $x = (x_1,\dots, x_d)\in \mathbb{R}^d$ we have $f(x_1,\ldots,x_d)=f(|x_1|,\ldots,|x_d|)$.
\end{definition}

\begin{definition} Let $f : \mathbb{R}^d\to [-\infty, \infty]$ be unconditional. Then $f$
is {\em block decreasing} if  the restriction of
$f$  to the non-negative cone $[0,\infty)^d$  is decreasing in each variable.
\end{definition}

\begin{remark} Observe that being unconditional  depends on the system of coordinates
 chosen for $\mathbb{R}^d$. A rotation, for instance, may destroy
this property.
The term ``block decreasing" comes from the statistical literature, while in functional
analysis ``unconditional decreasing" is used instead. Radial functions
  with respect to unconditional norms in $\mathbb{R}^d$ (for instance, the $\ell_p$-norms,
  $1\le p\le\infty$) are block decreasing. But in general, a block decreasing function
  need not have convex level sets, in which case it will not be radial with
  respect to any norm. On the other hand, a norm $\nu$ may fail to be
  unconditional; if so, a radial function with respect to $\nu$
  will
  not  be block decreasing.
\end{remark}

Given a norm $\mu$ in $\mathbb{R}^d$, we denote by $B_\mu(y,\delta) := \{x\in\mathbb{R}^d:
\mu (x - y) \le \delta\}$ the closed $\mu$-ball centered at
$y\in\mathbb{R}^d$ and of radius $\delta>0$. Absolute values around a
set denote its $k$-dimensional Lebesgue measure. While the dimension $k$
is not indicated in $|A|$, it will usually be clear from the context.
When doubts may arise, we explicitly  state what $k$ is.

\begin{definition}
Let  $f\in L^1_{loc}(\mathbb{R}^d)$. Then for all $ x\in\mathbb{R}^d$, the uncentered maximal function
 $M_\mu f$ is defined by
 \begin{equation}\label{maxdef}
    M_\mu f(x)=\sup_{\{y\in \mathbb{R}^d, \delta > 0:  x\in B_\mu(y,\delta)\}}\frac{1}{|B_\mu(y,\delta)|}\int_{B_\mu(y,\delta)}|f(u)|du.
  \end{equation}
Let $R > 0$ be fixed. The {\em local} maximal function $M_{R,\mu}f$ is defined by
imposing an extra condition on the radius of balls: $\delta \le R$. Other than that, the definition is identical to (\ref{maxdef}).
\end{definition}

We write $M_p$ instead of $M_\mu$ in the special case where $1\le p \le \infty$ and $\mu$ is an $\ell_p$-norm, i.e., given by
$\|x\|_p :=\left( |x_1|^p+ |x_2|^p+\dots + |x_d|^p\right)^{1/p}$ when
$1\le p < \infty$, and  by  $\|x\|_\infty :=\max_{1 \le i \le d} \left\{ |x_1|,  |x_2|, \dots , |x_d|\right\}$. Likewise, we write $B_p$ instead of
$B_\mu$ for balls.

\begin{remark} Since all
 norms on $\mathbb{R}^d$ are equivalent, maximal functions defined using
 different norms are always pointwise comparable. However, comparability  yields no information
about  regularity properties, or the size of derivatives, if they
exist in some appropriate sense.
\end{remark}

From now on, we assume that {\em all} functions appearing in this
paper are {\em locally integrable}, including functions of the form
$M_\mu f$. It might happen that for some
$f\in L^1_{loc} (\mathbb{R}^d)$, $M_\mu f \equiv\infty$. But then
 $M_\mu f$ is constant and thus, its variation is zero under any
 reasonable notion of variation. So we exclude this
 trivial case from any further consideration. It follows
 that $M_\mu f$ is finite almost everywhere.
 But in general, something else is needed to have local integrability
 of $M_\mu f$.  For the type of
 functions studied in this paper, that is, for block decreasing
 functions of bounded variation, the local integrability of
 $M_\mu f$ is easy to check (cf. Lemma \ref{locint} below);
but some auxiliary results are stated in greater generality,
and for these we {\em assume} local integrability of $M_\mu f$
from the start.

Let $\Omega\subset \mathbb{R}^d$ be   open. The following definition
 is taken from [AFP, p. 119].

\begin{definition}\label{defvar}
For $f\in L^1_{loc}(\Omega)$, the variation $V(f,\Omega)$ of $f$
  in $\Omega$ is given by
  \begin{equation*}
    V(f,\Omega):=\sup \left\{\int_\Omega f \operatorname{div}\phi dx: \phi \in [C^1_c(\Omega)]^d,
    \|\phi\|_\infty\le 1\right\}.
  \end{equation*}
\end{definition}

Suppose $V(f,\Omega) < \infty$, i.e.,
$f$ is of finite variation. If additionally $f\in L^1(\Omega)$, we write $f\in BV(\Omega)$, where $BV$
stands for bounded variation. Integration by parts shows that
if $f$ is continuously
differentiable in $\Omega$, then
$V(f,\Omega)=\int_\Omega|\nabla f|dx$. By Proposition 3.6,
p.120 of [AFP], $V(f,\Omega)<\infty$ if and only if there
exists an $\mathbb{R}^d$-valued Radon measure $Df=(D_1f,\ldots,D_df)$ on
$\Omega$ such that
\begin{equation*}
  \int_\Omega f \operatorname{div}\phi dx=-\int_\Omega \phi d
  Df \qquad \forall \phi \in [C^1_c(\Omega)]^d.
\end{equation*}
That is, the distributional derivative is representable by a
Radon measure $Df$ on $\Omega$ with
total variation
$|Df|(\Omega) < \infty$.
Furthermore, $|Df|(\Omega) = V(f,\Omega)$. The norm of
$f\in BV(\Omega)$ is defined by $\|f\|_{BV(\Omega)}:=
\|f\|_{L^1 (\Omega)} + |Df|(\Omega)$. Note that
$W^{1,1}(\Omega)\subset BV(\Omega)$, and the Sobolev norm on $W^{1,1}(\Omega)$
is simply the restriction to the latter space of the $BV$ norm.
Note also that a function $f\ge 0$ of bounded variation on
$\mathbb{R}^d$ need not
be bounded (provided $d\ge 2$, as we always assume in this paper);
 well known examples exist in $W^{1,1}(\mathbb{R}^d)$.
However, if $f$ is also block decreasing, then the hypothesis
$V(f,\mathbb{R}^d)<\infty$ entails that the {\em precise representative}
$f^*$  of $f$ (defined by taking the limsup in the Lebesgue Differentiation
Theorem)  must be finite except perhaps on a negligible $(d-1)$-dimensional
Hausdorff measurable set.  Since $f$ is block-decreasing, either $f^*(0) = \infty$ or $f^*$ is bounded.
But we must allow the possibility that $f^*(0) = \infty$,
so we will consider functions with values in $[0,\infty]$.
In general we do not assume that $f$ is integrable.

The first theorem of the paper states that the variation of $M_\mu f$
is controlled by the variation of $f$, and the same happens with the
local maximal function $M_{R,\mu} f$, with constant independent of $R$.

\begin{theorem}\label{t1}
  Let $f:\mathbb{R}^d\to [0,\infty]$ be a block decreasing function
  and let $\mu$ be an unconditional norm in $\mathbb{R}^d$. Then
 $V(M_\mu f,\mathbb{R}^d)\le c(\mu,d) V(f,\mathbb{R}^d)$, and,
with the same constant $c(\mu,d)$,
 $V(M_{R,\mu} f,\mathbb{R}^d)\le c(\mu,d) V(f,\mathbb{R}^d)$  for
  every $R > 0$.
\end{theorem}

 Theorem \ref{t2} states that the maximal function of a block
 decreasing function of finite
variation is continuous, except perhaps on a negligible $(d-1)$-dimensional
Hausdorff measurable set. The same happens with the local maximal
function. Observe that   unconditionallity of the norm is not
assumed here.

\begin{theorem}\label{t2}
  Let $f:\mathbb{R}^d\to [0,\infty]$ be a block decreasing function such that
   $V(f,\mathbb{R}^d)<\infty$. Let $\mu$ be a norm in $\mathbb{R}^d$,
   and let $R > 0$.
   Then $M_\mu f$ and $M_{R,\mu} f$ are continuous a.e. with respect to $\mathcal{H}^{d-1}$.
\end{theorem}

It is well known (see \cite[pp. 184--186]{AFP}) that if $f\in
L^1_{loc}(\mathbb{R}^d)$ and $V(f,\mathbb{R}^d)<\infty$, then the
distributional
derivative $Df$ of $f$  can be decomposed
in three parts,
\begin{equation*}
  Df=D^af+D^jf+D^c f,
\end{equation*}
where
 $D^af$ is
absolutely continuous, $D^jf$ is the {\em jump part} of $Df$,
its restriction to the jump set of $f$ (to be defined below, cf.
Definition \ref{jumpset}), and $D^cf$ is
the {\em Cantor part} of the measure, the singular part of
$Df$ that lives on the set where $f$ is
approximately continuous (cf. \cite[p. 160]{AFP} for the definition
of approximate continuity). The functions $f:\mathbb{R}^d\to [-\infty, \infty]$ of bounded variation for
which $D^cf=0$ are called functions
of {\em special
bounded variation}, and denoted by $SBV(\mathbb{R}^d)$. If both $D^cf$ and $D^jf$ vanish, then $f$
 is in the Sobolev space $W^{1,1}(\mathbb{R}^d)$, and
 for every measurable set $A$,
$Df (A) = \int_A \nabla f$.

 \begin{example} \label{square} In order to illustrate some notions that already
 have been defined and others that will appear later on, and also to explain the
 terminology, consider
 the following simple example. Let $f$ be
 the characteristic function  of the unit square $[-1/2, 1/2]^2\subset\mathbb{R}^2$. Then $f$ is a block decreasing, $BV$ function,
with $|Df| = |D^j f|$ and $|Df|(\mathbb{R}^2) = 4$, the length of the
boundary of the square. This boundary is also the jump set of $f$
(cf. Definition \ref{jumpset} below). And $|Df|= |D^j f|$ is just the linear Lebesgue measure
on the jump set. Since $D^cf=0$, $f$ is actually a $SBV$ function.
 The {\em centered} maximal function of $f$ has the
same jump set as $f$, though the jumps are smaller, and the uncentered
maximal function $M_\infty f$ associated to cubes has empty jump set.
For a general norm $\mu$, we know from the preceding theorem that
the jump set of  $M_\mu f$ has linear measure at most zero. In fact, it is easy
to see that the
jump set of  $M_\mu f$ contains, at most, the four corners of the square.
 \end{example}

The jump set of a function is obviously disjoint with the set of its continuity
points.  From the preceding theorem, together
with the fact that if $E\subset \mathbb{R}^d$ has $\mathcal{H}^{d-1}$-measure
zero then $D^j f (E) = 0$ (cf. \cite[Formula 3.90, p. 184]{AFP}),
 we obtain the following corollary. It says that essentially
(with respect to $\mathcal{H}^{d-1}$) $M_\mu f$ has no jumps.

\begin{corollary}\label{nojump}
  Let $f$ be a locally integrable, non-negative block decreasing function such that
   $V(f,\mathbb{R}^d)<\infty$. Let $\mu$ be a norm in $\mathbb{R}^d$,
   and let $R > 0$.
   Then $D^j M_\mu f = 0$, and $D^j M_{R,\mu} f = 0$.
\end{corollary}

 The integrability of $f$ is not assumed
 in the next result, so it deals
 with non-negative block decreasing functions slightly
 more general than those in $SBV(\mathbb{R}^d)$.

\begin{theorem}\label{sbv}
  Let $f:\mathbb{R}^d\to [0, \infty]$ be a block decreasing function with
   $V(f,\mathbb{R}^d)<\infty$ and $|D^c f| = 0$. Let $\mu$ be an
  unconditional norm in $\mathbb{R}^d$, and let $R > 0$. Then $M_\mu f$ has a weak gradient
  $\nabla M_\mu f$ in $L^1$, and there exists a constant $c(\mu,d) > 0$
  such that
  \begin{equation*}
    \int_{\mathbb{R}^d}|\nabla M_\mu f(x)|dx\le
    c(\mu,d) V(f,\mathbb{R}^d).
  \end{equation*}
The same result, with the same constant $c(\mu,d)$, holds for $M_{R,\mu} f$.
\end{theorem}

The preceding theorem gives a positive answer to \cite[Question 1]{HaOn} in the special case of
functions $f$ with $|f|$  block decreasing. In
fact, the condition $f\in W^{1,1}(\mathbb{R}^d)$ from \cite[Question 1]{HaOn} is relaxed (to
$V(f,\mathbb{R}^d)<\infty$ and $|D^c f| = 0$).

If $\mu$ happens to be the $\ell_\infty$-norm, then the ``no Cantor part" hypothesis
$|D^c f| = 0$ can be dispensed with. The reason for this
is that block decreasing functions are particularly well adapted to
arguments using  cubes, or more generally, rectangles with sides parallel
to the axes. Even though the next result is stated for cubes only,
it also holds for any norm defined using a fixed rectangle
(with sides parallel
to the axes).

\begin{theorem}\label{cubes}
Let $f:\mathbb{R}^d\to [0, \infty]$ be a block decreasing function
 such that $V(f,\mathbb{R}^d)<\infty$, let $M_\infty f$ be the
maximal function of $f$ defined using  cubes, i.e., $\ell_\infty$-balls, and let $R > 0$.
Then $M_\infty f$ has a weak gradient
$\nabla M_\infty f\in L^1(\mathbb{R}^d)$,
and
\begin{equation*}
    \int_{\mathbb{R}^d}\left|\nabla M_\infty f(x)\right|dx\le
    c_{d}V(f,\mathbb{R}^d).
  \end{equation*}
The same result holds for $M_{R,\infty} f$.
\end{theorem}

Denote the cone of non-negative block decreasing functions in $BV(\mathbb{R}^d)$, measured with the
$BV$ norm, by $BD (\mathbb{R}^d)$.
The preceding Theorems,  together with Theorem 2.19 and
Remark 2.20 of \cite{AlPe2}, entail the following boundedness
results on $BD (\mathbb{R}^d)$ for the local maximal operator.

\begin{corollary}\label{cor1}
 Given an unconditional norm $\mu$ on $\mathbb{R}^d$, the local maximal operator $M_{R,\mu}$ is bounded from
 $BD(\mathbb{R}^d)$ to $BV(\mathbb{R}^d)$. More precisely,
there exists a constant $c= c(\mu,d) > 0$ such that
 for all $R>0$ and all $f\in BD(\mathbb{R}^d)$ we have
\begin{equation}\label{BVnorm}
\|M_{R,\mu} f\|_{BV(\mathbb{R}^d)}\le
c\left(\|f\|_{BV(\mathbb{R}^d)} + \|f\|_{L^1 (\mathbb{R}^d)}\log^+ R \right).
\end{equation}
If $\mu=\|\cdot\|_\infty$, then $M_{R,\mu}$ is bounded from
 $BD(\mathbb{R}^d)$ to $W^{1,1}(\mathbb{R}^d)$, so
there exists a constant $c= c(\mu,d) > 0$ such that
 for all $R>0$ and all $f\in BD(\mathbb{R}^d)$
\begin{equation}\label{W11norm}
\|M_{R,\infty} f\|_{W^{1,1}(\mathbb{R}^d)}\le
c\left(\|f\|_{BV(\mathbb{R}^d)} + \|f\|_{L^1 (\mathbb{R}^d)} \log^+ R\right).
\end{equation}
\end{corollary}

Of course, if $R$ is fixed, (\ref{BVnorm}) reduces to
\begin{equation}\label{BVnorms}
\|M_{R,\mu} f\|_{BV(\mathbb{R}^d)}\le
c \|f\|_{BV(\mathbb{R}^d)},
\end{equation}
though perhaps with a different $c$, and likewise for (\ref{W11norm}), in the case
of cubes.

\section{The maximal function of a block decreasing function.}

In this  and the following sections, lemmas and proofs
will refer exclusively to the maximal operator $M$, since they are exactly the same for the
local operator $M_R$. The only exception occurs in Lemma \ref{bolasgrandes},
which is valid in the non-local case, under fewer hypotheses.
It has to do with the Lipschitz behavior of $Mf$ on some ``good sets",
for a locally integrable $f$.
Since as $R$ becomes small, $M_R f$ looks more like $f$, any improvement
of $M_R f$ over $f$ will tend to disappear as $R\to 0$.
Thus, the local case requires additional assumptions, and hence it is treated in a different lemma.

In this section we prove  that if $f$ is non-negative and block decreasing,
then $M_\mu f$ is block decreasing (see Lemma
\ref{lemma3} below; of course, $M_\mu f$ is always non-negative). First
we deal with the local integrability of $M_\mu f$.

\begin{lemma}
  \label{locint}
  Let $f: \mathbb{R}^d \to [0,\infty]$ be a block decreasing function
   such that
  $V(f,\mathbb{R}^d)<\infty$, and let $\mu$ be any norm in $\mathbb{R}^d$. Then $M_\mu f\in L^1_{loc} (\mathbb{R}^d)$.
\end{lemma}

\begin{proof} By pointwise comparability of maximal functions
associated to different norms, it is enough to prove the
result in the $\ell_\infty$ case. Fix a ball $B = B_\infty (0, r)$ (centered at the origin) and note that for any $x\in B$, to estimate
$M_\infty f (x)$ it suffices to average over cubes contained in $B$,
by the block decreasing property of $f$. Now by Sobolev embedding
for functions in $BV (B_\infty (0, r))$
(cf. for instance, \cite[Corollary 3.49, p. 152]{AFP}), we have
$f\in L^{\frac{d}{d - 1}}(B_\infty (0, r))$. Using the boundedness
of the maximal operator when $p > 1$, $M_\infty f\in L^{\frac{d}{d - 1}}(B_\infty (0, r))\subset L^{1}(B_\infty (0, r))$. Since $r$ is arbitrary, it follows that $M_\infty f\in L^1_{loc} (\mathbb{R}^d)$.
\end{proof}

On $\mathbb{R}$, being unconditional is the same as
being even, and being block decreasing, the same as being even and unimodal (decreasing on $(0,\infty)$).
So the next Lemma is trivial, and we omit the proof. We do mention that
a (nontrivial) higher dimensional version is known in the literature as
 Anderson's theorem, cf.
 \cite{An} or \cite[Theorem 11.1]{Ga}.

\begin{lemma}\label{bdl1}
  Let $f: \mathbb{R} \to [0,\infty]$ be  block decreasing.
Then for every $\delta >0$, the function
 $g_\delta(x):=\int_{x-\delta}^{x+\delta}f(u)du$ is block decreasing.
\end{lemma}

It follows from the definition of $V(f,\Omega)$ in terms of the distributional derivative
$Df$,  that functions equal a.e. have the same variation.
Since the measure $Df$ may have a singular part, it is
nevertheless
 useful to choose an everywhere defined representative of $f$.

\begin{definition}\label{precrep} Let $B_\mu$ denote a generic
 ball defined using the norm $\mu$. The {\em precise
representative} $f^*$ of $f$ is
\begin{equation*}
    f^*(x): =\limsup_{|B_\mu|\downarrow
    0, x\in B_\mu}\frac{1}{|B_\mu|}\int_{B_\mu}f(y)dy.
\end{equation*}
\end{definition}

The notation does not reflect the fact that $f^*$ depends on $\mu$,
since  this will make no difference in the arguments below.

A related notion of precise representative can be obtained
by taking the limsup over balls centered at $x$, instead of
balls containing $x$, as we do above.
With either choice of definition, it is not difficult to see that if $f$ is block decreasing, then so is $f^*$ (cf. Lemma \ref{precrep} below).

From now on, we  use the following notation. Let
$x=(x_1,\ldots,x_d)\in \mathbb{R}^d$,  let $i\in\{1,\ldots,d\}$,
and let $\{e_1,\ldots,e_d\}$ be the canonical basis of $\mathbb{R}^d$. We
denote by $\hat{x}_i$ the $(d-1)$-dimensional vector obtained from $x$
by removing its $i$-th component. That is,
$\hat{x}_{i}=(x_1,\ldots,x_{i-1},x_{i+1},\ldots,x_d)$. To shorten
expressions, we write $x=(\hat{x}_i,x_i)$. Even though the notation may suggest
otherwise, we do emphasize the fact that the
order of the coordinates in $x=(\hat{x}_i,x_i)$ is unaltered. We also write
$f(x) = f(\hat{x}_i,x_i)$ rather than $f(x) = f((\hat{x}_i,x_i))$.

\begin{lemma}\label{bdl2}
  Let $f:\mathbb{R}^d\rightarrow [0, \infty]$ be  block decreasing,
let $\mu$ be an unconditional norm in $\mathbb{R}^d$, and let $\delta>0$. Then
   $g(x):=\int_{B_\mu(x,\delta)}f(u)du$ is block
  decreasing.
\end{lemma}
\begin{proof} Since the argument is the same for each coordinate, we
focus on  the last one. It is enough to prove that if $x=(\hat{x}_d,x_d)\in [0, \infty)^d$ and   $h\ge0$, then
  $g(\hat{x}_{d},x_d+h)\le g(x)$. But this
follows from Fubini's Theorem: Write
\begin{equation*}
  P:=\{\hat{u}_d\in\mathbb{R}^{d-1}: \mbox{ there exists a real number } u_d \mbox{ such that }(\hat{u_d}, u_d)\in B_\mu(0,\delta)\},
\end{equation*}
and
\begin{equation*}
  S_{\hat{u}_d}=\{t \in \mathbb{R}: u= (\hat{u_d}, t)\in B_\mu(0,\delta)\},
  \end{equation*}
that is, $ S_{\hat{u}_d}$ is the vertical section
in $B_\mu(0,\delta)$ associated to $\hat{u}_d$.
The assumption that $\mu$ is  unconditional entails that each
$S_{\hat{u}_d}$ is an interval centered at 0.
 Now
\begin{equation*}
  g(x)=\int_{B_\mu(0,\delta)}f(u+x)du= \int_P
  \left(\int_{S_{\hat{u}_d}}f(u+x)
  du_d\right)d\hat{u}_d
\end{equation*}
and
\begin{equation*}
  g(x+he_d)=\int_{B_\mu(0,\delta)}f(u+x+he_d)du= \int_P
  \left(\int_{S_{\hat{u}_d}}f(u+x+he_d)
  du_d\right)d\hat{u}_d.
\end{equation*}
For fixed $\hat{u}_d$ and $\hat{x}_d$ the function
$f(\hat{u}_d+\hat{x}_d,\cdot)$ is block decreasing, so by Lemma
\ref{bdl1},
\begin{equation*}
  \int_{S_{\hat{u}_d}}f(\hat{u}_d +\hat{x}_d,u_d+x_d+h)
  du_d \le  \int_{S_{\hat{u}_d}}f(\hat{u}_d +\hat{x}_d,u_d+x_d)
  du_d,
\end{equation*}
and now $g(\hat{x}_{d},x_d+h)\le g(x)$ follows by putting together the last three equations.
\end{proof}

\begin{lemma}\label{precrep} Let $\mu$ be an unconditional norm in $\mathbb{R}^d$.
  If $f:\mathbb{R}^d\rightarrow [0, \infty]$ is  block decreasing,
  then so is its precise representative $f^*$.
  \end{lemma}

\begin{proof}  It is enough to prove that if $x=(\hat{x}_d,x_d)\in [0, \infty)^d$ and   $h > 0$, then
  $f^*(\hat{x}_{d},x_d+h)\le f^* (x)$. But this
follows from the previous lemma, applied to any ball $B_\mu (a, \delta)$
containing $(\hat{x}_{d},x_d+h)$ and with radius $\delta < \mu ((h/2) e_d)$.
\end{proof}

\begin{lemma}\label{lemma3}
  Let $f:\mathbb{R}^d\rightarrow [0,\infty]$ be  block decreasing
and let the norm $\mu$ be  unconditional. Then,
   $M_\mu f$ is a block
  decreasing function.
\end{lemma}
\begin{proof}
  It is enough to prove that if $x=(\hat{x}_d,x_d)\in [0, \infty)^d$ and   $h\ge0$, then
  $M_\mu f(\hat{x}_d,x_d+h) \le M_\mu f(x)$, and to conclude this it
  suffices to show that given an arbitrary ball $B_\mu(a,\delta)$
  containing $x+he_d$, we have
  \begin{equation}\label{bdl3e1}
    \frac{1}{|B_\mu(a,\delta)|}\int_{B_\mu(a,\delta)}f(u)du
    \le M_\mu f(x).
  \end{equation}
Let us see why.  Since $x+he_d\in B_\mu(a,\delta)$,
by unconditionallity
$(\hat{x}_d, 0)\in B_\mu((\hat{a}_d, 0),\delta)$.
Now if $a_d \le x_d$, then $x_d\in [a_d, x_d+h)$, so
$x\in B_\mu(a,\delta)$ by convexity of the ball. Thus,  (\ref{bdl3e1}) holds in this case. And if $a_d > x_d$, then
$x\in B_\mu((\hat{a}_d,x_d),\delta)$, and
(\ref{bdl3e1}) follows from Lemma \ref{bdl2}, since we just lowered the
ball in the vertical direction.
\end{proof}

\section{Controlling the variation of block decreasing functions.}

The purpose of this section is to find a quantity equivalent to
the variation (cf. Definition \ref{defvar}), and easier to compute for block decreasing functions.
Denote by $f_{\hat{x}_i}$ the one-dimensional function
$f_{\hat{x}_i}(x_i):=f(x)$. We use the fact that finite variation can be characterized via the variation along the coordinate axes (cf. \cite[p. 196]{AFP}, or \cite[\S5.10]{EvGa}). Suppose $f$ is $C^1$. Integrating
pointwise the $\ell_1$-norm and the $\ell_2$-norm of its gradient, and using
$\|\cdot\|_2\le \|\cdot\|_1\le \sqrt{d}\|\cdot\|_2$, we obtain
\begin{equation}\label{derdirec}
  \int_{\mathbb{R}^d}|\nabla
f(u)|du\le \sum_{i=1}^d \int_{\mathbb{R}^d}|D_if(u)|du \le
\sqrt{d}\,\int_{\mathbb{R}^d}|\nabla f(u)|du,
\end{equation}
where $D_i f$ denotes the partial derivative of $f$ with respect to
$x_i$, i.e., the derivative of $f_{\hat{x}_i}$.
Since for a continuously differentiable $f$ we have
$V(f,\mathbb{R}^d)=\int_{\mathbb{R}^d}|\nabla
f(u)|du$, and for each fixed $\hat{x}_i$
we have $V(f_{\hat{x}_i},\mathbb{R}^d)=
 \int_{\mathbb{R}}|D_i f_{\hat{x}_i} (t)| dt$, (\ref{derdirec}) and
 an approximation
argument show that
\begin{equation}\label{vardirec}
  V(f,\mathbb{R}^d)\le \sum_{i=1}^d
  \int_{\mathbb{R}^{d-1}}V(f_{\hat{x}_i},\mathbb{R})d\hat{x}_i \le
\sqrt{d}\,
  V(f,\mathbb{R}^d),
\end{equation}
and this formula also holds when $V(f,\mathbb{R}^d) =\infty$.
From (\ref{vardirec}), it follows that for a block decreasing function
\begin{equation}\label{equivar}
 V(f,\mathbb{R}^d) \le  2 \sum_{i=1}^d
  \int_{\mathbb{R}^{d-1}}
  [f_{\hat{x}_i}(0^+)-f_{\hat{x}_i}(\infty)]d\hat{x}_i
  \end{equation}
\begin{equation}\label{equivar1}
 = 2^d \sum_{i=1}^d
  \int_{[0,\infty)^{d-1}}
  [f(\hat{x}_i,0^+)-f(\hat{x}_i,\infty)]d\hat{x}_i
  \le \sqrt{d} V(f,\mathbb{R}^d),
\end{equation}
where $f(\hat{x}_i,0^+):=\lim_{t\rightarrow 0^+}f(\hat{x}_i,t)$ and
$f(\hat{x}_i,\infty):=\lim_{t\rightarrow\infty}f(\hat{x}_i,t)$.

Next we show that the value at
infinity of a block decreasing function of finite variation is the
same in essentially all directions. Exceptions may occur, though, if
at least one coordinate remains fixed at $0$.

\begin{lemma}
  \label{varlim1}
  Let $f: \mathbb{R}^d \to [0,\infty]$ be a block decreasing function
   such that
  $V(f,\mathbb{R}^d)<\infty$. Then $\inf_{\mathbb{R}^d} f =
\lim_{t\rightarrow \infty}
  f(t,\ldots,t)$. Furthermore,
for every $x = (x_1,\ldots,x_{d}) \in (0,\infty)^d$, and
all $i\in \{1\dots, d\}$,
  \begin{equation}\label{limites}
    \inf_{\mathbb{R}^d} f =\lim_{t\rightarrow \infty}
f(\hat{x}_i,t).
\end{equation}
Additionally,
\begin{equation}\label{limmax}
    \inf_{\mathbb{R}^d} f = \lim_{t\rightarrow \infty}
M_\mu f(\hat{x}_i,t).
\end{equation}
\end{lemma}

\begin{proof} Given $x  \in (0,\infty)^d$, let $t = \max\{x_1,\ldots,x_{d}\}$. Since $f$ is block decreasing, $f(x) \ge
  f(t,\ldots,t)$, so
  $\inf_{\mathbb{R}^d} f =
\lim_{t\rightarrow \infty}
  f(t,\ldots,t)$, where the limit exists by monotonicity.

Next, we may assume, by symmetry,  that $i=d$. Note that
if (\ref{limites}) fails for a fixed
$\hat{x}_d = (x_1,\ldots,x_{d-1})$,  then it
fails for all $\hat{y}_d \in (0,x_1]\times\ldots \times(0,x_{d-1}]$
by monotonicity in each variable. Thus, it suffices
to prove (\ref{limites}) for $(x_1,\ldots,x_{d-1})$ in a
full measure subset of
$(0,\infty)^{d-1}$.
 We use induction on the dimension $d$.
The result is obvious for $d=1$, so we assume it holds for $d-1$
and show
 that
 (\ref{limites}) also holds  for $d\ge 2$.
   Consider the  function of $d-1$ variables
$f_{x_1}(\cdot):=f(x_1,\cdot)$. Clearly, $f_{x_1}$ is a
block decreasing function. By (\ref{equivar}) and (\ref{equivar1}), for almost all
$x_1>0$ we have $V(f_{x_1},\mathbb{R}^{d-1})<\infty$. Using induction,
we apply (\ref{limites}) to $f_{x_1}$ and
conclude that for every $x_2,\ldots,x_{d-1}>0$,
\begin{equation*}
  \lim_{t\to \infty}f_{x_1}(t,\ldots,t,t)= \inf_{\mathbb{R}^{d - 1}} f_{x_1}
  = \lim_{t\to
  \infty}f_{x_1}(x_2,\ldots,x_{d-1},t).
\end{equation*}
Now if $\inf_{\mathbb{R}^d} f < \inf_{\mathbb{R}^{d - 1}} f_{x_1}$,  there exists an $N>0$ such that for all $t'\ge N$, and all $w_2,\dots , w_d > 0$,
  \begin{equation*}
    f(x_1, w_2,\ldots, w_d)-f(t',t',\ldots,t') \ge
    \frac{\inf_{\mathbb{R}^{d - 1}} f_{x_1} - \inf_{\mathbb{R}^d} f}{2}.
  \end{equation*}
  Thus, using (\ref{equivar1}) we derive the following contradiction:
  \begin{equation*}
    \infty>V(f,\mathbb{R}^d)\ge
    \int_{[N,\infty)^{d-1}}[f(\hat{z_1},0^+)-f(\hat{z}_1,\infty)]d\hat{z}_1\ge
  \end{equation*}
  \begin{equation*}
    \ge
    \int_{[N,\infty)^{d-1}}[f(\hat{z_1},x_1)-f(N,\ldots,N)]d\hat{z}_1  \end{equation*}
\begin{equation*}
\ge \frac{\inf_{\mathbb{R}^{d - 1}} f_{x_1} - \inf_{\mathbb{R}^d} f}{2}
\int_{[N,\infty)^{d-1}} d\hat{z}_1  =\infty.
\end{equation*}
Therefore  $\inf_{\mathbb{R}^d} f = \inf_{\mathbb{R}^{d - 1}} f_{x_1}$,
and (\ref{limites}) follows. To obtain (\ref{limmax}), note that the local integrability of $f$, together
with the existence of a limit at infinity, entail that averages will
approach this limit as $t\to\infty$.
\end{proof}

\section{Variation of the maximal function.}

We are now ready to prove Theorem \ref{t1}.

\begin{proof}
It is clear that if we add a constant to a
function, its variation does not change.
Recalling that $f$ is locally integrable and non-negative,
it is easy to check that
 $M_\mu f(x)=M_\mu
(f-\inf_{\mathbb{R}^d}f)(x)+ \inf_{\mathbb{R}^d}f$, so for simplicity we
suppose that
 $\inf_{\mathbb{R}^d}f=0$. Under the assumptions
 and with the notation of
 Lemma \ref{varlim1}, this entails
  that
 $\lim_{t\rightarrow \infty}
M_\mu f(\hat{x}_i,t)  =0$ (even if
$f\notin L^1 (\mathbb{R}^d)$).
Now by  (\ref{equivar}) and (\ref{equivar1}) it is enough to
check that
\begin{equation*}
  \sum_{i=1}^d\int_{[0, \infty)^{d-1}}
  M_\mu f(\hat{x}_i,0^+)d\hat{x}_i\le c(\mu,d)\sum_{k=1}^d
  \int_{[0, \infty)^{d-1}}
  f(\hat{x}_k,0^+)d\hat{x}_k.
\end{equation*}
Because of the equivalence of all norms on $\mathbb{R}^d$, and in particular,
between $\mu$ and $\|\cdot\|_\infty$, it suffices to
consider the maximal operator defined by cubes and to prove that
\begin{equation}\label{desigualdaddenormas}
  \sum_{i=1}^d\int_{[0, \infty)^{d-1}}
  M_\infty f(\hat{x}_i,0^+)d\hat{x}_i\le c(d)\sum_{k=1}^d
  \int_{[0, \infty)^{d-1}}
  f(\hat{x}_k,0^+)d\hat{x}_k.
\end{equation}
Assume that $i=d$. We want to estimate
\begin{equation*}
  \int_{[0, \infty)^{d-1}}
  M_\infty f(y,0^+)dy.
\end{equation*}
To this end, we divide $[0, \infty)^{d-1}$ in $d!$ suitable subsets
as follows: Denote by $\mathcal{P}_{n}$ the set of all permutations
of $n$ elements. For each
  $\sigma \in \mathcal{P}_{d-1}$ we define
\begin{equation*}
  A_\sigma= \{y\in [0, \infty)^{d-1} : y_{\sigma(1)}\ge
  y_{\sigma(2)}\ge \ldots \ge y_{\sigma(d-1)}\}.
\end{equation*}
Then $\bigcup_{\sigma \in \mathcal{P}_{d-1}}A_\sigma =
[0, \infty)^{d-1}$. By symmetry, it is enough to
consider  the
identity permutation, the other estimates being the same.
So we take
\begin{equation}\label{asigma}
  A_\sigma= \{y\in [0, \infty)^{d-1} : y_1\ge
  y_{2}\ge \ldots \ge y_{d-1}\}.
\end{equation}
Fix $y\in A_\sigma$. To estimate $M_\infty f(y,0)$, let $B_\infty (a, k)$
be a cube (i.e., an $\ell_\infty$-ball)   containing $(y,0)$.
Set $p^{k,y}:=(\max\{0,y_1-k\},\ldots,\max\{0,y_{d-1}-k\},0)$, and
note that
 $0\le p^{k,y}_i\le |a_i|$  for every $i=1,\ldots,d -1$. Thus,
by Lemma \ref{bdl2} we have
\begin{equation*}
  M_\infty f(y,0)=\sup_{k>0}\frac{1}{|B_\infty (p^{k,y}, k)|}\int_{B_\infty (p^{k,y}, k)}f(u)du.
\end{equation*}
We introduce the auxiliary endpoints $y_0:=\infty$ and
$y_d:= 0$. With this notation,
for every $k>0$ there exists a $j\in\{0,\ldots,d-1\}$ such that
\begin{equation}\label{ordencoordenadas}
y_0 \ge y_1\ge \ldots\ge y_j\ge 4k\ge y_{j+1}\ge\ldots\ge
y_{d-1}\ge y_d.
\end{equation}

Note that
$q:=(y_1-k,\ldots,y_j-k, 0,\ldots,0)\in [0,\infty)^d$
satisfies $q\le p^{k,y}$ under the partial order induced
by the coordinates, i.e., for all $i=1,\dots, d$, $q_i\le p_i^{k,y}$.
Since $ p_i^{k,y} \le |a_i|$ also,
by Lemma \ref{bdl2}, the average of a block decreasing function over
$B_\infty (a, k)$ is smaller than the average over
$B_\infty (q, k)$. Let $u\in B_\infty (q, k)$ be arbitrary.
By (\ref{ordencoordenadas}),
\begin{equation*}
 f(u)\le
f(y_1-2k,\ldots,y_j-2k,u_{j+1},\ldots,u_{d})\le
f\left(\frac{y_1}{2},\ldots,\frac{y_j}{2},u_{j+1},\ldots,u_{d}\right).
\end{equation*}
 Thus,
\begin{equation*}
 \frac{1}{|B_\infty (a, k)|}\int_{B_\infty (a, k)} f(u)du\le \frac{1}{|B_\infty(q,k)|}\int_{B_\infty(q,k)}
 f(u)du
\end{equation*}
\begin{equation*}
 \le \frac{1}{(2k)^{d-j}}\int_{[-k,k]^{d-j}}f\left(\frac{y_1}{2},\ldots,\frac{y_j}{2},u_{j+1},\ldots,u_{d}\right)du_{j+1}\cdots
  du_d
\end{equation*}
\begin{equation}\label{intermedio}
  \le \frac{1}{k^{d-j}}\int_{[0,k]^{d-j}}f\left(\frac{y_1}{2},\ldots,\frac{y_j}{2},u_{j+1},\ldots,u_{d}\right)du_{j+1}\cdots
  du_d.
\end{equation}
Observe that for averages of block decreasing functions over cubes
centered at the origin, the smaller the radius, the greater
the average. Thus, if $j\le d-2$, the term in
(\ref{intermedio}) is bounded by
\begin{equation*}
\left(\frac{4}{y_{j+1}}\right)^{d-j}\int_{\left[0,\frac{y_{j+1}}{4}\right]^{d-j}}f\left(\frac{y_1}{2},\ldots,\frac{y_j}{2},u_{j+1},\ldots,u_{d}\right)du_{j+1}\cdots
  du_d,
\end{equation*}
while if $j=d-1$, then the term in
(\ref{intermedio}) is bounded by $f(y/2,0)$.
 Hence, for every $y\in A_\sigma$,
\begin{equation*}
  Mf(y,0)\le \max\left\{f\left(\frac{y}{2},0\right),\max_{0\le j\le
  d-2}\left(\frac{4}{y_{j+1}}\right)^{d-j}G_j(y_1,\ldots,y_j,y_{j+1})\right\},
\end{equation*}
where for $1\le j \le d - 2$, $G_j$ is defined via
\begin{equation*}
  G_j (y_1,\ldots,y_j,y_{j+1}):=\int_{\left[0,\frac{y_{j+1}}{4}\right]^{d-j}}f\left(\frac{y_1}{2},\ldots,\frac{y_j}{2},u_{j+1},\ldots,u_{d}\right)
  du_{j+1}\cdots
  du_d,
\end{equation*}
and if $j = 0$, we just use the same formula, but integrating
$f(u_1,\dots, u_d)$ (i.e., no $y_i$ appears as an argument of $f$).

Thus,
\begin{equation}\label{upper}
  \int_{A_\sigma}Mf(y,0)dy \le
  \int_{[0,\infty)^{d-1}}  f\left(\frac{y}{2},0\right) dy\,
  +\sum_{j=0}^{d-2}\int_{A_\sigma}\left(\frac{4}{y_{j+1}}\right)^{d-j}
  G_j(y_1,\ldots,y_j,y_{j+1})dy.
\end{equation}
Note that for $j\in\{0,\ldots,d-2\}$ (see \ref{asigma}),
\begin{equation}\label{bigger}
A_\sigma\subset \{y\in [0,\infty)^{d-1} : y_{j+1}\ge
  y_{j+2}\ge \ldots \ge y_{d-1}\},
\end{equation}
and
\begin{equation}\label{simplex}
\int_0^{y_{j+1}}\left( \cdots\left( \int_0^{y_{d-2}}dy_{d-1}\right)
\cdots
\right) dy_{j+2} =\frac{y_{j+1}^{d-j-2}}{(d-j-2)!}.
\end{equation}

Using Fubini's Theorem and the definition of $G_j$ we have
\begin{equation}\label{fubini}
\int_{[0,\infty)^{j+1}}
  \frac{G_j(y_1,\ldots,y_j,y_{j+1})}{y_{j+1}^2}dy_1\cdots dy_{j+1}
\end{equation}
\begin{equation} \label{fubini2}
=
\int_{[0,\infty)^{j+1}}
  \frac{1}{y_{j+1}^2}\int_{\left[0,\frac{y_{j+1}}{4}\right]^{d-j}}f\left(\frac{y_1}{2},\ldots,\frac{y_j}{2},u_{j+1},\ldots,u_{d}\right)
  du_{j+1}\cdots
  du_d dy_1\cdots dy_{j+1}
 \end{equation}
 \begin{equation} \label{fubini3}
\le
\int_{[0,\infty)^{d}}
 \left(f\left(\frac{y_1}{2},\ldots,\frac{y_j}{2},u_{j+1},\ldots, u_{d}\right) \int_{\max\{u_{j+1},\ldots, u_{d}\}}^\infty \frac{dy_{j+1}}{y_{j+1}^2}\right)
 dy_1\cdots dy_{j} du_{j+1}\cdots
  du_d
 \end{equation}
 \begin{equation} \label{fubini4}
=
\int_{[0,\infty)^{d}} \frac{f(\frac{y_1}{2},\ldots,\frac{y_j}{2},u_{j+1},\ldots,u_{d})}{\max\{u_{j+1},\ldots,u_d\}}
 dy_1\cdots dy_{j} du_{j+1}\cdots
  du_d.
 \end{equation}
 The assumption $j\le d -2$ is used in the next application
 of Fubini's Theorem. To prove the bound
 \begin{equation}\label{fubini5}
\int_{[0,\infty)^d}
  \frac{f(u)}{\max\{u_{j+1},\ldots,u_d\}}
   du\le
\sum_{k=j+1}^d\int_{[0,\infty)^{d-1}}f(\hat{u}_k,0)
   d\hat{u}_k,
\end{equation}
split $[0,\infty)^d$ into the regions  where each
$u_k$ is the maximum value. Suppose, for instance,
that we are considering
$E_d := [0,\infty)^d\cap \{\max\{u_{j+1},\ldots,u_d\} = u_d\}$.
Now $j + 1\le d - 1$, so we can select an $u_k$ with $k\ne d$
and replace $u_k$ with $0$. Since $f$ is block decreasing,
\begin{equation*}
\int_{E_d}
  \frac{f(u)}{u_d}
   du\le
\int_{[0,\infty)^{d-1}}f(\hat{u}_k,0)\int_0^{u_d} \frac{1}{u_d} d u_k \,
   d\hat{u}_k = \int_{[0,\infty)^{d-1}}f(\hat{u}_k,0)
   d\hat{u}_k,
\end{equation*}
and (\ref{fubini5}) follows.
Next, using (\ref{upper}--\ref{fubini5}), we obtain
\begin{equation*}
  \int_{A_\sigma}Mf(y,0)dy \le 2^{d-1}
  \int_{[0,\infty)^{d-1}}f(y,0)dy
\end{equation*}
\begin{equation*}
 + \sum_{j=0}^{d-2}\frac{4^{d-j}}{(d-j-2)!}\int_{[0,\infty)^d}
\frac{f(\frac{y_1}{2},\ldots,\frac{y_j}{2},u_{j+1},\ldots,u_{d})}{\max\{u_{j+1},\ldots,u_d\}}
  dy_1\cdots dy_j du_{j+1}\cdots du_d=
\end{equation*}
\begin{equation*}
=2^{d-1}
  \int_{[0,\infty)^{d-1}}f(y,0)dy\,+ \sum_{j=0}^{d-2}\frac{2^{2d-j}}{(d-j-2)!}\int_{[0,\infty)^d}
  \frac{f(u)}{\max\{u_{j+1},\ldots,u_d\}}
   du\le
\end{equation*}
\begin{equation*}
 \le 2^{d-1}
  \int_{[0,\infty)^{d-1}}f(y,0)dy\,+  \sum_{j=0}^{d-2}\frac{2^{2d-j}}{(d-j-2)!}\sum_{k=j+1}^d\int_{[0,\infty)^{d-1}}f(\hat{u}_k,0)
   d\hat{u}_k\le
\end{equation*}
\begin{equation*}
\le    c^\prime (d)\sum_{k=1}^d\int_{[0,\infty)^{d-1}}f(\hat{u}_k,0)
   d\hat{u}_k .
\end{equation*}

Finally, (\ref{desigualdaddenormas}) follows by applying the same
estimate to each of the the $d!$ regions $A_\sigma$ and adding up.

\end{proof}

\begin{remark}
Recalling  inequalities (\ref{derdirec}) and (\ref{vardirec}),
it is natural to define a ``partial variation" for each variable $x_i$:
\begin{equation*}
  V_i(f,\mathbb{R}^d):=\int_{\mathbb{R}^{d-1}}V(f_{\hat{x}_i},\mathbb{R})d\hat{x}_i.
\end{equation*}
We have seen that the variation of $f$ controls the variation
of $M_\mu f$ (Theorem \ref{t1}). Here we show that the
partial variations of $f$ do not individually control the corresponding
partial variations of $M_\mu f$, something that makes the proof
of Theorem \ref{t1} harder than would otherwise be.
To see that the
inequality
  \begin{equation*}
    V_i(M_\mu f,\mathbb{R}^d)\le c(\mu, d)V_i(f,\mathbb{R}^d)
  \end{equation*}
may fail, consider the following counterexample in the case
$\mu = \|\cdot\|_\infty$.
Let $g$ be a non increasing function on $[0,\infty)$ such that
$g(0)=1$ and $g(\infty)=0$. Suppose also that $\|g\|_1=1$ and
$\int_0^\infty g(u)du/u=\infty$.
For $x_1,x_2\in\mathbb{R}$ and $m\in\mathbb{N}$ we define the block decreasing functions
$f_m(x_1,x_2)=mg(m|x_1|)g(|x_2|)$.
Then $V_2(f_m,\mathbb{R}^2)=4\int_0^\infty
f_m(x_1,0)dx_1=4g(0)\int_0^\infty mg(mx_1)dx_1=4g(0)\|g\|_1=4$.
On the other hand, for $x_1>0$,
\begin{equation*}
M_\infty f_m(x_1,0)\ge \frac{1}{x_1^2}\int_0^{x_1} mg(mu_1)du_1\int_0^{x_1}
g(u_2)du_2=
\end{equation*}
\begin{equation*}
= \frac{1}{{x_1}^2}\int_0^{mx_1} g(u_1)du_1\int_0^{x_1} g(u_2)du_2=:
F_m(x_1).
\end{equation*}
Now
$\lim_{m\rightarrow\infty}F_m(x_1)=\|g\|_1\int_0^{x_1}g(u_2)du_2/x_1^2=\int_0^{x_1}g(u_2)du_2/x_1^2$,
so, using monotone convergence and the Fubini-Tonelli theorem, we obtain
\begin{equation*}
\lim_{m\rightarrow\infty}V_2(M_\infty f_m,\mathbb{R}^2)\ge\lim_{m\rightarrow\infty}\int_0^\infty
  F_m (x_1) dx_1=
\end{equation*}
\begin{equation*}
=   \int_0^\infty
 \left(\int_0^{x_1}g(u_2)du_2\right) \frac{dx_1}{x_1^2}
   =\int_0^\infty g(u_2)\frac{du_2}{u_2}=\infty.
\end{equation*}

It is easy to check that this
example  can be adapted to  $M_\mu f$, where  $\mu$ is any  unconditional norm.
\end{remark}

\section{Continuity of the maximal function.}

In this section we prove Theorem \ref{t2},
showing that if $f$ is a
 block decreasing function of  finite variation, then
 $M_\mu f$ is continuous, except perhaps on a negligible $(d-1)$-dimensional
Hausdorff measurable set.

First we recall the notion of
approximate continuity (see \cite[p. 47 and 209]{EvGa}). Note that
the definitions and results from \cite{EvGa} are given in terms
of euclidean balls, so we need to use
 the equivalence of all norms in $\mathbb{R}^d$.

\begin{definition}
  Let $f:\mathbb{R}^d\rightarrow \mathbb{R}$. We say that $l$ is the
  approximate limit of $f$ as $y\to x$, and write
  \begin{equation*}
    \operatorname{ap}\lim_{y\to x}f(y)=l,
  \end{equation*}
  if for each $\varepsilon >0$,
  \begin{equation*}
    \lim_{r\to 0}\frac{|B_2(x,r)\cap\{|f-l|\ge
    \varepsilon\}|}{|B_2(x,r)|}=0.
  \end{equation*}
\end{definition}
That is, if $l$ is the approximate limit of $f$ at $x$, for all
$\varepsilon >0$ the sets $\{|f-l|\ge
    \varepsilon\}$ have density zero at $x$.

\begin{definition}\label{fsup}
  Let $f:\mathbb{R}^d\rightarrow \mathbb{R}$. We say that $f_{sup}(x)$ is the
  approximate lim sup of $f$ as $y\to x$ if
  \begin{equation*}
    f_{sup}(x):=\operatorname{ap} \limsup_{y\to x}f(y)=\inf
    \left\{t : \lim_{r\to 0}\frac{|B_2(x,r)\cap\{f>t\}|}{r^d}=0\right\}.
  \end{equation*}
Likewise,  $f_{inf}(x)$ is the
  approximate lim inf of $f$ as $y\to x$ if
  \begin{equation*}
    f_{inf}(x):=\operatorname{ap} \liminf_{y\to x}f(y)=\sup
    \left\{t : \lim_{r\to 0}\frac{|B_2(x,r)\cap\{f<t\}|}{r^d}=0\right\}.
  \end{equation*}
\end{definition}

It is well known that for measurable functions
the approximate limit exists a.e.
\cite[Theorem 3, p.47]{EvGa}. For locally integrable functions, this follows
from the Lebesgue Differentiation Theorem. For locally integrable
functions of finite variation,  the approximate limsup and
liminf are finite $\mathcal{H}^{d-1}$ a.e. on $\mathbb{R}^d$ (cf.
\cite[Theorem 2, p.211]{EvGa}; actually, the results
from \cite{EvGa} are stated
for $BV$ functions, so $f\in L^1$. But it is easy to check that
local integrability of $f$ suffices to carry out the arguments).

\begin{definition}\label{jumpset}
The {\em jump set} $J_f:=\{f_{inf}(x)<f_{sup}(x)\}$ of $f$ is the set of points where the
approximate limit of $f$ does not exist.
\end{definition}

\begin{definition}
  Let $v$ be a unit vector in $\mathbb{R}^d$ and let $x\in\mathbb{R}^d$. We
  define the half-spaces associated to $x$ and $v$ by
  \begin{equation*}
    H_v^+:=\{y\in\mathbb{R}^d:v\cdot(y-x)\ge 0\},
  \end{equation*}
    \begin{equation*}
    H_v^-:=\{y\in\mathbb{R}^d:v\cdot(y-x)\le 0\},
  \end{equation*}
where the symbol $\cdot$ denotes the usual scalar product in $\mathbb{R}^d$.
\end{definition}

While the notation does not make it explicit, $H_v^+$ and $H_v^-$ depend
on $x$ and contain it as a boundary point. Set

$$F(x) := \frac{f_{sup}(x)+f_{inf}(x)}{2}.$$

From \cite[p.213, Theorem 3]{EvGa} we get
\begin{theorem}\label{densidadvspromedio}
If $f\in L_{loc}^1(\mathbb{R}^d)$ and
  $V(f,\mathbb{R}^d)<\infty$, then,
  \\
    i) for $\mathcal{H}^{d-1}$ a.e. $x\in \mathbb{R}^d-J_f$,
  \begin{equation*}
  \lim_{r\to 0^+}\frac{1}{|B_2(x,r)|}\int_{B_2(x,r)}|f(y)-F(x)|^{d/(d-1)}dy=0,
 \end{equation*}
  and\\
  ii) for $\mathcal{H}^{d-1}$ a.e. $x\in J_f$, there exists a unit
  vector $v\equiv v(x)$ such that
  \begin{equation*}
    \lim_{r\to 0^+} \frac{1}{|B_2(x,r)|}\int_{B_2(x,r)\cap
    H_v^-}|f(y)-f_{sup}(x)|^{d/(d-1)}dy=0,
  \end{equation*}
  and
\begin{equation*}
    \lim_{r\to 0^+} \frac{1}{|B_2(x,r)|}\int_{B_2(x,r)\cap
    H_v^+}|f(y)-f_{inf}(x)|^{d/(d-1)}dy=0.
  \end{equation*}
\end{theorem}

In the preceding Theorem, euclidean balls
can be replaced by balls defined using any other norm $\mu$,
noting that integrands are non-negative, and then
giving up some constant. Of course, the corresponding limits are
still zero.
So for $\mathcal{H}^{d-1}$ almost every
$x\in \mathbb{R}^d-J_f$, $F= f^*$, the precise representative of $f$
defined using $\mu$.

A different  definition of approximate limits
and related notions appears
in \cite{AFP} (using integral averages, in the line of
the preceding theorem).
But all such definitions coincide
 $\mathcal{H}^{d-1}$ a.e. with the ones given above, so it is actually
 immaterial which ones we use.

Since approximate limit exist for a.e. $x\in
\mathbb{R}^d$,  all the functions $f^*$, $F$, $f_{inf}$ and
$f_{sup}$ represent the same equivalence class $[f]$. To study  the continuity of the maximal function it will be convenient for us to use
$f_{sup}$.

\begin{lemma}
  If $f$ is block decreasing on $\mathbb{R}^d$, then
  $f_{sup}$ is block decreasing.
\end{lemma}

\begin{proof}
 Fix $x_1,\ldots,x_d\ge 0$, $h>0$, and $t > 0$. We  prove that $f_{sup}(x)\ge
  f_{sup}(x+he_d)$, the argument being the same for the other coordinates.
  By Definition \ref{fsup}, it is enough to show that for all sufficiently
  small $r >0$, $|B_2(x+he_d,r)\cap\{f>t\}|\le| B_2(x,r)\cap\{f>t\}|$.
Suppose $0<r<h/2$. Given
$y\in B_2(x+he_d,r)\cap\{f>t\}$, we have $y - h e_d \in B_2(x,r)$, and by the choice of $r$,
$y_d\ge |y_d-h|$, so $t < f(y)\le f(y-he_d)$. Thus
$\left(B_2(x+he_d,r)\cap\{f>t\}\right) - h e_d \subset  B_2(x,r)\cap\{f>t\}$,
and $|B_2(x+he_d,r)\cap\{f>t\}|\le| B_2(x,r)\cap\{f>t\}|$ follows by the translation invariance of Lebesgue measure.
\end{proof}

The next lemma  states that in any half-ball resulting from the intersection of an euclidean ball
$B_2 (x,r)$ with a half-space having $x$ in its boundary, there is a
comparable $\mu$-ball contained in the half-ball and containing
$x$ (as a boundary point, of course).

\begin{lemma}\label{conl3}
  Let
  $\mu$ be an arbitrary norm on $\mathbb{R}^d$, let $r>0$, and let
  $x,v\in\mathbb{R}^d$, where $\|v\|_2=1$. Then there exists a constant $k_\mu>0$ such
   that for every half-ball $B_2 (x,r)\cap H_v^+$,
  we can find a center $c\in\mathbb{R}^d$ and a radius $\rho>0$ with
  \begin{equation*}
    x\in B_\mu(c,\rho)\subset B_2(x,r)\cap H_v^+ \quad\textnormal{ and
    }\quad\frac{|B_\mu(c,\rho)|}{|B_2(x,r)|}\ge k_\mu.
  \end{equation*}
\end{lemma}
\begin{proof}
By a translation we may assume that $x = 0$. Let $B_\mu (0,\rho)$
be the largest $\mu$-ball contained in $B_2 (0,r/4)$. By the convexity
of $B_\mu (0,\rho)$, it can be translated, say, to $B_\mu (c,\rho)$,
in such a way that $0$ belongs to the boundary of $B_\mu (c,\rho)$
and this ball is contained in $H_v^+$. Since
$0\in B_\mu (c,\rho)\subset B_2 (c, r/4) \subset B_2 (0, r)$,
it follows that $B_\mu (c,\rho)\subset B_2 (0,r)\cap H_v^+$. Furthermore,
if $t > 0$, then $|B_\mu (c, t \rho)| = t^d |B_\mu (c, \rho)|$,
by the scaling properties of Lebesgue measure. Using the fact that
all norms in $\mathbb{R}^d$ are equivalent, we let $t>0$
be the smallest real number such that $B_2 (0, r)\subset B_\mu (0, t \rho)$,
and conclude that $|B_2 (0, r)| \le t^d |B_\mu (c, \rho)|$. Then we take $k_\mu = t^d$.
\end{proof}

\begin{lemma}\label{conl2}
  Let  $0 \le f\in L^1_{loc}(\mathbb{R}^d)$, let
   $V(f,\mathbb{R}^d)<\infty$, and let
  $\mu$ be a norm.
  Then for
  $\mathcal{H}^{d-1}$ almost every $x\in\mathbb{R}^d$, we have $M_\mu f(x)\ge
  f_{sup}(x)$.
\end{lemma}
\begin{proof} For every $r>0$,
  \begin{equation*}
    M_\mu f(x)\ge \frac{1}{|B_\mu(x,r)|}\int_{B_\mu(x,r)}f=
    f_{sup}(x) +
    \frac{1}{|B_\mu(x,r)|}\int_{B_\mu(x,r)}(f(y)-f_{sup}(x))dy,
  \end{equation*}
so it is enough to show that for
  $\mathcal{H}^{d-1}$ almost every $x\in\mathbb{R}^d$,  the right
  most term in the preceding inequality tends to $0$ as
  $r\rightarrow 0^+$. If $x\in\mathbb{R}^d\setminus J_f$, then part
$i)$  of Theorem \ref{densidadvspromedio} yields the result: First,
replace $\mu$-balls by euclidean balls (perhaps giving up some constant) and
then use Jensen's inequality. And if $x\in J_f$, a similar argument, using
part $ii)$  of Theorem \ref{densidadvspromedio} and Lemma \ref{conl3}, yields the same result.
\end{proof}

\begin{lemma}\label{fsuppos}
  Let $f$ be a block decreasing function on $\mathbb{R}^d$, and let
  $x\in (0,\infty)^d$. Then $f_{sup}$ is upper semicontinuous
at $x$.
\end{lemma}
\begin{proof}
Suppose otherwise. Since $f_{sup}$ is block decreasing,  there
is an $\varepsilon>0$
such that for all $z$ in the rectangle $
  \Pi_{i=1}^d(0,x_i)$, $f_{sup}(z)>f_{sup}(x)+\varepsilon$.
  Then the density at $x$ of the set
  $\{f>f_{sup}(x)+\varepsilon\}$ is at least $1/2^d$,
  contradicting the definition of $f_{sup}$.
\end{proof}

\begin{remark}\label{corollary}
Let $f$ be  block decreasing and let $x\in (0,\infty)^d$ be
such that $f(x)=f_{sup}(x)$. Arguing as in Lemma
\ref{fsuppos}, we conclude that $f$ is upper semicontinuous at $x$.
Thus, a block decreasing
function is upper semicontinuous at almost every point in $\mathbb{R}^d$.
\end{remark}

\begin{remark}
  \label{fsupupcontae}
 By Lemma
\ref{fsuppos}, if $f$ is block decreasing and
 $f_{sup}$ is not upper semicontinuous  at $x$, then
  at least one of $x$'s coordinates must be zero.
We consider these points  in the next Lemma.
\end{remark}

\begin{lemma}\label{lem0}
  Let $f$ be a block decreasing function on $\mathbb{R}^d$. Then
  for almost all  $(x_1,\ldots,x_{d-1})\in\mathbb{R}^{d-1}$, $f_{sup}$ is upper
  semicontinuous  at $(x_1,\ldots,x_{d-1},0)$.
\end{lemma}
\begin{proof}
Writing $\hat{x}_d=(x_1,\ldots,x_{d-1})$,
  $\hat{y}_d=(y_1,\ldots,y_{d-1})$, and
  $g(\hat{y}_d):=f_{sup}(\hat{y}_d,0)$, we note,
first, that by Remark \ref{corollary}, $g$ is upper semicontinuous for a.e.
  $\hat{x}_d\in\mathbb{R}^{d-1}$. Thus, it is enough to check that if $g$ is
  upper semicontinuous at $\hat{x}_d$, then $f_{sup}$ is upper
  semicontinuous at $(\hat{x}_d,0)$. But this follows from the
  block decreasing property of $f_{sup}$:
$\limsup_{y\to (\hat{x}_d,0)} f_{sup}(y) \le \limsup_{y\to (\hat{x}_d,0)}
f_{sup} (\hat{y}_d,0) = \limsup_{\hat{y}_d\to \hat{x}_d}
g(\hat{y}_d) \le g(\hat{x}_d) = f_{sup}(\hat{x}_d,0)$.
\end{proof}

{\em Proof of Theorem \ref{t2}.} According to
Lemma 3.4. of \cite{AlPe}, if a locally integrable function $h\ge0$
is upper semicontinuous at $w$ and $h(w) \le M h(w)$, then $Mh$ is
continuous at $w$. Now
 Lemmas \ref{fsuppos} and \ref{lem0} entail  that
 $f_{sup}$
 is upper semicontinuous at $\mathcal{H}^{d-1}$ almost every point,
 while by Lemma \ref{conl2},  $f_{sup}(w) \le M_\mu f_{sup} (w)$
 for
$\mathcal{H}^{d-1}$ a.e. $w$. Since $M_\mu f = M_\mu f_{sup}$,  the result follows.
\qed

\begin{remark}
  Note that the maximal function of a block decreasing function
  need not be continuous. Consider, for instance,
  the maximal function $M_\infty$, defined using cubes,
and let $f$ be the  characteristic function of the unit $\ell^p$-quasiball
in $\mathbb{R}^2$,
where $0<p\le1$ is fixed. Then $M_\infty f$ is discontinuous at $(1,0)$.
\end{remark}

\section{The derivative of the maximal function.}

Let us recall a few facts about the distributional derivative
of a function of finite variation (cf. \cite[p.184--186]{AFP}). The measure $Df$ vanishes on $\mathcal{H}^{d-1}$-negligible sets. Its
absolutely continuous part  $D^af$ is obtained by
integrating the density of
$Df$ with respect to the $d$-dimensional Lebesgue measure, and lives
in the set $\mathcal{D}_f$ where $f$ is approximately differentiable
(cf. \cite[Definition 3.70, p. 165]{AFP} for the definition of
approximate differentiability). The singular part of $Df$ can
be decomposed into a Cantor part and a jump part. The Cantor part $D^cf$
 gives
full measure
to the set where $f$ is
approximately continuous, and the jump part $D^jf$ gives
full measure
to the jump set of $f$ (a
countably $\mathcal{H}^{d-1}$-rectifiable set). The measure
 $D^af+D^cf$
 vanishes on sets of finite (and thus of $\sigma$-finite) $\mathcal{H}^{d-1}$
measure.

Let $f\in L^1_{loc}(\mathbb{R}^d)$.
From now on, we  assume that $|f|=|f|^*$, the precise
representative of $|f|$. In what follows, Lipschitz constants are determined by
using the $\ell_2$-norm.

We define
\begin{equation}\label{enk}
  E_{n,k}:=\{x\in\mathbb{R}^n:  \textnormal{ there exists a ball } B:= B_\mu(c,r) \textnormal{ with  } x\in B ,  r\ge 1/n,
  \end{equation}
  \begin{equation*}
  \frac{1}{|B|}\int_{B} |f(y)|dy =  M_{\mu}f(x), \textnormal{ and } M_{\mu}f(x) \le  k\}.
\end{equation*}

\begin{lemma}\label{bolasgrandes} Let $f$ be a locally integrable
function, and let $c_\mu > 0$ be such that
for all $w\in \mathbb{R}^d$,
$\|w\|_\mu \le c_\mu \|w\|_2$.
  Then the restriction of $M_\mu f$  to $E_{n,k}$
  is Lipschitz, with the Lipschitz constant $\operatorname{Lip}(M_{\mu} f) \le c_\mu dkn$.
  \end{lemma}

\begin{proof}
  Let $x,y\in E_{n,k}$. By symmetry, we may assume
  that $M_{\mu} f(x)\ge M_{\mu} f(y)$.
     Suppose  $M_{\mu} f(x) =   \frac{1}{|B|}\int_B |f|$, where $x\in B:= B_\mu (c,r)$.
Now since $x\in B_\mu(c,r)$, $y\in B_\mu(c,r+\|x-y\|_\mu)$. Thus,
  \begin{equation*}
    \frac{|M_{\mu}f(x)-M_\mu f(y)|}{\|x-y\|_2}= \frac{M_{\mu}f(x)-M_\mu f(y)}{\|x-y\|_2}
  \end{equation*}
  \begin{equation*}
   \le  \frac{M_{\mu}f(x)-\frac{| B_\mu (c,r)|}{|B_\mu(c,r+\|x-y\|_\mu)|} \frac{1}{| B_\mu (c,r)|}\int_{B_\mu (c,r)} |f|}{\|x-y\|_2}
   \end{equation*}
  \begin{equation*}
  \le \frac{M_{\mu}f(x)}{\|x-y\|_2}\left(1-\frac{|B_\mu(c,1)|r^d}{|B_\mu(c,1)|(r+\|x-y\|_\mu)^d}\right)
  \end{equation*}
  \begin{equation*}
   \le   \frac{c_\mu k}{\|x-y\|_\mu}\left(1-\frac{1}{\left(1+\frac{\|x-y\|_\mu}{r}\right)^d}\right)
   \le
   \frac{c_\mu k}{\|x-y\|_\mu}\left(1-\frac{1}{ (1+n\|x-y\|_\mu)^d}\right).
  \end{equation*}
Now $g(a) := a^{-1}\left(1- (1+na)^{-d}\right)$ is decreasing on $\{a>0\}$, and $\lim_{a\to 0^+}g(a)=d n$, so   from the preceding inequality we obtain
  \begin{equation*}
    \frac{|M_{\mu}f(x)-M_\mu f(y)|}{\|x-y\|_2}\le  c_\mu k d n.
  \end{equation*}
  \end{proof}

Next we deal with the local maximal function $M_{R,\mu}$.
Set
\begin{equation}\label{ern}
  E_{R, n}:=\{x\in\mathbb{R}^n:  \textnormal{ there exists a ball } B:= B_\mu(c,r) \textnormal{ with  } x\in B , r\in [1/n, R],
  \end{equation}
\begin{equation*}
 \textnormal{ and }
  \frac{1}{|B|}\int_{B} |f(y)|dy =  M_{R,\mu}f(x)\}.
\end{equation*}
Of course, if $R <  1/n$, then $E_{R, n}$ is empty.

\begin{lemma}\label{bddballs} Fix $R > 0$, and let $c_\mu := 1/|B_\mu (0,1)|$.  If $f$ is locally integrable and
has finite variation, then  the restriction of $M_{R,\mu} f$ to $E_{R,n}$
  is Lipschitz, with the Lipschitz constant $\operatorname{Lip}(M_{R,\mu} f) \le c_\mu n^d V(f, \mathbb{R}^d)$.
    \end{lemma}

\begin{proof} As noted above,  we may  assume that
  $R \ge 1/n$. Otherwise $E_{R, n} = \emptyset$ and there is
  nothing to prove. So let $x,y\in E_{R, n}$, and
  suppose
  that $M_{\mu} f(x) =   \frac{1}{|B|}\int_B |f| \ge M_{\mu} f(y)$,  where $x \in B:= B_\mu (c,r)$ and  $1/n\le r \le R$.
 It follows from \cite[Exercise 3.3, p. 208]{AFP}
  that for every bounded measurable set $K$,
  \begin{equation*}
    \int_{K} \frac{|f(x + h) - f(x)|}{\|h\|_2} dx
    \le  V(f, \mathbb{R}^d).
  \end{equation*}
  Thus,
    \begin{equation*}
    \frac{|M_{R, \mu}f(x)-M_{R, \mu} f(y)|}{\|x-y\|_2}
    \le
       \frac{\frac{1}{|B_\mu (c,r)|}\int_{B_\mu (c,r)} |f| -\frac{1}{|B_\mu(c + y - x,r)|}\int_{B_\mu(c + y - x,r)} |f|}{\|x-y\|_2}
 \end{equation*}
  \begin{equation*}
 \le
\frac{1}{|B_\mu (c,r)|} \int_{B_\mu (c,r)} \frac{|f (u ) - f(u - x + y)| }{\|x-y\|_2} du \le c_\mu n^d  V(f, \mathbb{R}^d).
   \end{equation*}
  \end{proof}

Since the exact size of the Lipschitz constants is irrelevant in
the argument that follows,  from now on we only consider  $M_\mu$.

The following Lemma appears in \cite[p.75]{EvGa}.
We mention that in \cite{EvGa}, the same notation is used
for Hausdorff measures
and the outer measures they generate; in particular, the result below
applies to arbitrary sets $E$.

  \begin{lemma}\label{lemato}
    Let $f:\mathbb{R}^d\rightarrow \mathbb{R}$ be a Lipschitz function
with Lipschitz constant $\operatorname{Lip}(f)$ and let $s > 0$.
     Then, for all $E \subset\mathbb{R}^d$ we have
     $\mathcal{H}^s (f(E))\le \operatorname{Lip}(f)^s \mathcal{H}^{s}(E)$.
  \end{lemma}

Define $E:=\bigcup_{n,k\in\mathbb{N}}E_{n,k}$, and note that
$\mathbb{R}^d-E\subset \{M_{\mu}f=f^*\}$. Now it is to be expected
 that $DM_\mu f$
has no Cantor part on  $E$, since
  $M_\mu f$ Lipschitz on the sets $E_{n,k}$, and likewise, that
$DM_\mu f$ has no Cantor part on  $\{M_{\mu}f=f^*\}$, since
by hypothesis $f^*$ is of $SBV$, so $D^c f \equiv 0$. We prove
that this is indeed the case, by restricting
functions to
lines.

 \begin{lemma}\label{lemmatres}
  If $f\in L^1_{loc}(\mathbb{R}^d)$  has finite variation, then $M_\mu f$ maps $\mathcal{H}^{1}$-negligible subsets of $E$ into $\mathcal{H}^{1}$-negligible sets.
\end{lemma}
\begin{proof}  Fix $k$ and $n$, and let $N\subset \mathbb{R}^d$ be an $\mathcal{H}^{1}$-null subset of $\mathbb{R}^d$.
  By
  the previous lemma, $|M_\mu f(N \cap E_{n,k})| = 0$;
  here the absolute value signs stand for the 1-dimensional Lebesgue
  measure, which on the real line coincides with $\mathcal{H}^{1}$; while the preceding lemma refers to Lipschitz functions
  defined on all $\mathbb{R}^d$, one can always extend a
  Lipschitz function from a subset to the whole space $\mathbb{R}^d$,
  with the same constant by Kirszbraun's theorem,
  or, if one is not concerned about the
  constant, as is our case, by simpler extension theorems.
  Since a countable union of null sets is
  null, the result follows.
  \end{proof}

We shall use a variant of the Banach Zarecki Theorem (which states that a
real valued continuous function on a compact interval is
absolutely continuous
if and only if it is  of bounded variation and maps null sets to null sets). As stated, the result fails for $\mathbb{R}$ even if $f$ is bounded; for
instance, the function $\sin x$ is absolutely continuous and has
infinite variation. However, under the additional assumption that $f\ge 0$ is  block decreasing, the variation is bounded by $2 f(0)$,
so the following version of the Banach Zarecki Theorem does
hold
for $\mathbb{R}$.

\begin{lemma}\label{BVAC}
 Let $f:\mathbb{R} \to [0, \infty)$ be a continuous, block decreasing function. Then $f$  is absolutely continuous if and only if
 $f$  maps measure zero sets
to measure zero sets.
\end{lemma}

We use the following notation to express the  decomposition of a function $h$ on
$\mathbb{R}^d$,
 into functions $h_j(x'; t)$ defined on lines. For every $j=1,\ldots,d$ and every $x'=(x'_1,\ldots,x'_{d-1})\in
  \mathbb{R}^{d-1}$ we  set
  $h_j(x';t):=h(x'_1,\ldots,x'_{j-1},t,x'_j,\ldots,x'_{d-1})$.  This decomposition of $h$ leads to the corresponding disintegration
  result for $Dh$ (cf. \cite{AFP}). Regarding the Cantor part $D^c h$ of
$Dh$, it follows from \cite[Theorem 3.108]{AFP} that
it can be recovered from the Cantor parts of the derivatives of
the restrictions of $h$ to lines, and in particular, $D^c h = 0$
if and only if for almost every line parallel to the $j$-th coordinate axis,
$D^c h_j = 0$, where $j=1,\ldots,d$. To show that on these lines the functions $h_j(x'; t)$ map Lebesgue null sets to
Lebesgue null sets, we modify them so they become continuous, and then apply the preceding lemma.

Recall
  that, in order to simplify notation, we are assuming that
  a function $f\ge 0$ is everywhere equal to its precise representative
$f^*$.

\begin{lemma} \label{fn0to0}
  Let $f: L^1_{loc}(\mathbb{R}^d)\to [0,\infty]$ be a  finite variation,
  block decreasing function with $|D^cf|(\mathbb{R}^d)=0$. Then for $\mathcal{H}^{d-1}$ a.e. $x'\in \mathbb{R}^{d-1}$, and for every $j=1,\ldots,d$, the function $f_j(x';\cdot):\mathbb{R}\to \mathbb{R}$ maps $\mathcal{H}^{1}$-negligible sets into $\mathcal{H}^{1}$-negligible sets.
\end{lemma}

\begin{proof} Fix $\varepsilon > 0$ and $j\in \{1,\dots ,d\}$.
By (\ref{vardirec}),
for a.e. $x'\in\mathbb{R}^{d-1}$,
$f_j(x';\cdot)$ is a function with finite variation, since
$V(f, \mathbb{R}^d) < \infty$, and by
\cite[Theorem 3.108]{AFP},
for a.e. $x'\in\mathbb{R}^{d-1}$,
$f_j(x';\cdot)$ has null Cantor derivative.
Denote by  $A$ the subset of all $x'\in\mathbb{R}^{d-1}$ for which
both of the preceding conditions hold.
Note in particular that $f_j(x';\cdot) < \infty$
on $A$. For each  $x^\prime\in A$, $f_j(x^\prime;\cdot)$ is a non-negative,
real valued function with at most a countable number of jump discontinuities. Next we modify $f_j(x^\prime;\cdot)$ so it
becomes continuous.

Suppose $\{t_n\}_{n=1}^\infty$ is a listing of the set where
$f_j(x^\prime;\cdot)$ has jumps. Since both the right and left limits
$f_j(x^\prime; t_n +)$ and $f_j(x^\prime; t_n -)$ exist, there is
a $\delta_n > 0$ making the images under $f_j(x^\prime; t)$ of the intervals
$[t_n - \delta_n, t_n]$ and $[t_n, t_n + \delta_n]$
 so small
 that $|f_j(x^\prime; [t_n - \delta_n, t_n + \delta_n])| < \varepsilon/2^n$. Thus,
$|\cup_n f_j(x^\prime; [t_n - \delta_n, t_n + \delta_n])| < \varepsilon$.
Actually, we want to select
$\delta_{n}$ satisfying some additional conditions that will make it
possible to obtain
 disjoint intervals.
Using countability we assume that $\delta_n$ is chosen so $t_n - \delta_n$ and $t_n + \delta_{n}$ do not
belong to the jump set of  $f_j(x^\prime;\cdot)$.
We define inductively  the sequence of disjoint intervals as
follows, starting with $n_1 = 1$: Given $[t_1 - \delta_1, t_1 + \delta_1]$, let
$t_{n_2}$ be the first jump point in the list not belonging to
$(t_1 - \delta_1, t_1 + \delta_1)$, if there is any (otherwise
stop here). Then choose $ \delta_{n_2}$, satisfying the conditions
above, so that additionally $[t_1 - \delta_1, t_1 + \delta_1]
\cap [t_{n_2} - \delta_{n_2}, t_{n_2} + \delta_{n_2}] = \emptyset$. Now repeat the process, letting $t_{n_3}$ be the first jump point not
contained in the union of the two preceding intervals, and so on. Once the process
stops, we define $h (t) := f_j (x^\prime; t)$ on $\mathbb{R}\setminus
\cup_{n_k} [t_{n_k} - \delta_{n_k}, t_{n_k} + \delta_{n_k}]$, and on each
$[t_{n_k} - \delta_{n_k}, t_{n_k} + \delta_{n_k}]$, we extend $h(t)$  affinely  from $f_j(x^\prime; t_{n_k} - \delta_{n_k})$
to $f_j(x^\prime; t_{n_k} + \delta_{n_k})$. Now, let
$N\subset \mathbb{R}$ be null. Then $|h(N)|= 0$ by Lemma \ref{BVAC},
so $|f_j (x^\prime; N)| \le |h(N)| + | f_j(x^\prime; \cup_{n_k} [t_{n_k} - \delta_{n_k}, t_{n_k} + \delta_{n_k}])| < \varepsilon$, and
since $\varepsilon$ is arbitrary, we conclude that
$|f_j (x^\prime; N)| = 0$.
\end{proof}

{\em Proof of Theorem \ref{sbv}.} It is enough to show that $M_\mu f$ is absolutely
continuous on lines ($ACL$), that is, given any coordinate axis,
say $x_d$ without loss of generality, $M_\mu f$ is absolutely
continuous on almost all lines parallel to the the $x_d$-axis, where
almost all refers to Lebesgue measure on the intersection of the lines with the
$(d-1)$-subspace perpendicular to them. Sobolev Theory then
entails that $M_\mu f\in W_{loc}^{1,1} (\mathbb{R}^d)$, so the
distributional gradient $\nabla M_\mu f$  exists as a locally integrable (vector valued) function.
Furthermore, since by Theorem \ref{t1}
\begin{equation*}
  V(M_\mu f,\mathbb{R}^d)\le c_{\mu,d} V(f,\mathbb{R}^d),
  \end{equation*}
from the fact that the variation of $f$ is finite, and under the
assumption that $M_\mu f$ is $ACL$, we
obtain
\begin{equation*}
\int_{\mathbb{R}^d} |\nabla M_\mu f| = |D M_\mu f| (\mathbb{R}^d) =  V(M_\mu f,\mathbb{R}^d)\le c_{\mu,d} V(f,\mathbb{R}^d) < \infty,
\end{equation*}
so $\nabla M_\mu f  \in L^1(\mathbb{R}^d)$. Next we show
 that $M_\mu f$ is $ACL$.

Recall that $E:= \cup_{n,k} E_{n,k}$, where $E_{n,k}$ is given by
(\ref{enk}). Consider the set of all lines
 parallel to the last coordinate $x_d$ (of course, any other coordinate
 axis will do equally well). Let $A$ be the set of all
 $y\in \mathbb{R}^{d-1}$
 such  that $f_d (y;\cdot)$ is a function of finite variation
and $M_\mu f(y, t)$ is a continuous function of $t$. By Theorem \ref{t2}
and by (\ref{vardirec}), $|A^c| = 0$;
here $|\cdot|$ stands for the $(d-1)$-dimensional Lebesgue measure.
Given $y\in A$, denote by $L_y$ the vertical line $\{(y,t):t\in\mathbb{R}\}$.
Let  $N_y\subset L_y$ be null with respect to the $1$-dimensional Lebesgue measure, denoted in what follows by $|\cdot|$. By Lemma \ref{fn0to0}, for almost all
$y\in A$, $|f (N_y)| = 0$.
Now  $|M_\mu f (N_y\cap E)|= 0$ by Lemma \ref{lemmatres}, and
$E^c\subset \{M_\mu f = f\}$, so
for almost all $y\in A$,
$|M_\mu f (N_y \cap E^c)| \le |f (N_y)| = 0$.
Thus, $|M_\mu f (N_y)|= 0$ for almost every $y\in \mathbb{R}^{d-1}$.
  \qed

\vskip .2 cm

{\em Proof of Theorem \ref{cubes}.}
As in the previous argument, it is enough to show that $M_\infty f$ is $ACL$. We shall see that the ``no Cantor part of the derivative" hypothesis
is not needed here, since we are using cubes and $f$ is block decreasing.

Let $x\in(0,\infty)^d$ (the argument is the same for the other octants) and suppose, without loss of generality, that  $x_1=\min_{1\le i\le d}x_i$.
 Since $f$ is block decreasing, it is easy to see that
 in order to compute
 $M_\infty f(x)$, it is enough to consider cubes $B_\infty (c, r)$ of sidelength at least
 $x_1$.
 Thus, $[x_1,\infty)^d \subset E_{n,k}$
 for $n \ge 1/ x_1$ and $k \ge M_\infty f(x_1, x_1, \dots, x_1)$
 (the set  $E_{n,k}$ was defined in (\ref{enk})). So $M_\infty f$
 is Lipschitz on every open set compactly contained in $(0,\infty)^d$.
It follows that if $L$ is any line parallel to the $x_d$-coordinate
axis and $N\subset L$ is $1$-null, then $M_\infty f (N\cap (0,\infty)^d)$
is $1$-null. Since the same result holds for the intersection of $N$
with any other open octant, we can conclude that $|M_\infty f (N)| = 0$, unless for some
$i = 1,\dots, d - 1$, $L\subset \{x_i=0\}\times\mathbb{R}$.
But this is a null set of lines, so  $M_\infty f$ is $ACL$
by Theorem \ref{t2} and Lemma \ref{BVAC}.
\qed

\end{document}